\documentclass[leqno,11pt]{amsart}
\usepackage[colorlinks,pagebackref,hypertexnames=false]{hyperref}

\usepackage{amsmath,amsthm,amssymb,mathrsfs} 
\usepackage[alphabetic,backrefs]{amsrefs}
\usepackage{ae,aecompl}
\usepackage[margin=1in]{geometry} 
\usepackage[retainorgcmds]{IEEEtrantools}
\usepackage{soul}

\usepackage[english]{babel}

\usepackage{bookmark}

\usepackage{amsmath,amsfonts,amssymb,graphics,graphicx,latexsym,amscd,subfigure,hyperref}
\usepackage[usenames]{xcolor}
\usepackage[retainorgcmds]{IEEEtrantools}
\usepackage{graphicx}
\usepackage{amssymb}
\usepackage{xypic}
\usepackage{hyperref}
\usepackage{amsmath,amssymb}
\usepackage{dsfont}

\numberwithin{equation}{section}

\DeclareMathOperator\tr{\mathrm{tr}}
\DeclareMathOperator\Hess{\mathrm{Hess}}
\DeclareMathOperator\del{\partial}

\DeclareMathOperator\dvol{dvol}

\newcommand{\bbR}{\mathbb{R}}

\newcommand{\bbT}{\mathbb{T}}
\newcommand{\bbC}{\mathbb{C}}
\newcommand{\bbD}{\mathbb{D}}

\newcommand{\ip}{\cdot}

\newcommand{\mcC}{\mathcal{C}}
\newcommand{\mcH}{\mathcal{H}}
\newcommand{\bcp}{\mathbb{CP}}

\newcommand{\inte}{\mathrm{int}}

\newtheorem{theorem}{Theorem}

\newtheorem{corollary}{Corollary}

\newtheorem{proposition}{Proposition}
\newtheorem{question}{Question}
\theoremstyle{definition} \newtheorem{definition}{Definition}
\newtheorem{example}{Example}

\newtheorem{remark}{Remark}

\author{Gon\c{c}alo Oliveira}
\address[Gon\c{c}alo Oliveira]{Universidade Federal Fluminense IME--GMA, Niter\'oi, Brazil}
\urladdr{\href{https://sites.google.com/view/goncalo-oliveira-math-webpage/home}{sites.google.com/view/goncalo-oliveira-math-webpage/home}}
\email{\href{mailto:galato97@gmail.com}{galato97@gmail.com}}

\author{Rosa Sena-Dias} 
\address[Rosa Sena-Dias]{Instituto Superior T\'ecnico}
\email{rsenadias@math.ist.utl.pt}

\title[]{Minimal Lagrangian tori and action-angle coordinates} 
\thanks{This work was partially supported by FCT/Portugal through project PTDC/MAT-GEO/1608/2014, by Funda\c{c}\~ao Serrapilheira 1812-27395, by CNPq grants 428959/2018-0 and 307475/2018-2, and FAPERJ through the program Jovem Cientista do Nosso Estado.}

\begin{document}

\begin{abstract}
We investigate which orbits of an $n$-dimensional torus action on a $2n$-dimensional toric K\"ahler manifold $M$ are minimal. In other words, we study minimal submanifolds appearing as the fibres of the moment map on a toric K\"ahler manifold. Amongst other questions we investigate and give partial answers to the following: (1) How many such minimal Lagrangian tori exist? (2) Can their stability, as critical points of the area functional, be characterised just from the ambient geometry? (3) Given a toric symplectic manifold, for which sets of orbits $S$, is there a compatible toric K\"ahler metric whose set of minimal Lagrangian orbits is $S$?
\end{abstract}

\maketitle

\tableofcontents

\section{Introduction}

\subsection{Context}
A toric manifold is a symplectic manifold $(M^{2n}, \omega)$ endowed with an effective, Hamiltonian action of a torus of maximal dimension $\bbT^n$. The orbit space for such an action can be identified with a subset $P$ of Euclidean space via a moment map 
$$\mu : M \to P\subset\bbR^n\simeq (\text{Lie}(\bbT^n))^*.$$ 
If $\mu$ is proper, then $P$ is a convex polytope known as the moment polytope. The interior fibres of $\mu$ are Lagrangian $n$-tori. Toric manifolds admit an abundance of K\"ahler metrics which are invariant under the torus action. We refer to the toric manifold endowed with such a K\"ahler metric $g$ as a toric {\emph {K\"ahler}} manifold and denote it by $(M,\omega,g,\mu)$. 

On a Riemannian manifold, a submanifold is said to be minimal if its mean curvature vanishes identically, which is the condition that the submanifold be critical for the area functional. Minimal submanifolds of Riemannian manifolds carry a lot of geometric information and there has been great effort spent in both finding examples and studying properties of such objects. In dimension one, these minimal submanifolds are geodesics and play a crucial role in Riemannian Geometry. Even for familiar manifolds such as the round sphere or complex projective space with the Fubini-Study metric, we still do not have a complete understanding of higher dimensional minimal submanifolds despite of the large amount of work that has been put into the subject. 
K\"ahler manifolds carry a compatible integrable complex structure and their complex submanifolds are minimal. However, as K\"ahler manifolds are also symplectic, another natural class of submanifolds to consider is that of Lagrangians. The problem of minimising area in Hamiltonian isotopy classes of Lagrangians has played an important role in both Riemannian and Symplectic Geometries as well as in Mathematical Physics \cite{Schoen,Thomas,Thomas-Yau,Liu}.

A possible way to find examples of minimal submanifolds is by making use of ambient symmetries. In that direction, Palais' principle of symmetric criticality \cite{Palais} has been extensively used in finding examples of critical points of functionals using symmetry techniques. This principle, roughly states that symmetric configurations at which a functional is critical for all symmetric variations are actually critical in the strong sense, i.e. among all variations. This has been made rigorous in the context of minimal submanifolds by Hsiang-Lawson \cite{hl}. It has since been known that symmetry can play a simplifying role in our understanding of minimal submanifolds and this has been used in several contexts to find examples (see \cite{hl}, \cite{Pacini} and \cite{k}). In a certain sense, toric manifolds are maximally symmetric and it is therefore quite natural, in the spirit of Hsiang-Lawson, to take advantage of such symmetry to study minimal submanifolds. In \cite{Goldstein}, Goldstein studies torus invariant minimal Lagrangian tori on {\it K\"ahler-Einstein} toric manifolds and shows there is a unique minimal torus fibre. In \cite{Pacini}, Pacini recovers Goldstein's result; generalizes the method to manifolds with non-Abelian group actions and studies mean curvature flow in this context. In \cite{lr}, Legendre and Rollin study Hamiltonian-minimal submanifolds in the toric context. Our goal in this paper is, roughly speaking, to study minimal submanifolds of toric manifolds using the framework of action-angle coordinates and symplectic potentials developed in \cite{gu} and \cite{a}. Some of our results extend those in \cite{Goldstein} but we are interested in general toric K\"ahler metrics which need not be canonical in any sense. In other words, our results apply to all toric K\"ahler manifolds. 

On a toric manifold $(M^{2n}, \omega, \mu)$, the fibres of the moment map are of smaller dimension along the boundary $\partial P$ of the moment polytope $P$. Thus, for any compatible metric $g$, the $n$-dimensional area of the fibres is a non-negative function 
$$V:P \to \mathbb{R},$$ 
which vanishes along $\partial P$ and so must have a strict interior maximum. The corresponding $\bbT^n$-orbit turns out to be a minimal Lagrangian torus and other such may or not exist depending on the geometry of $(M^{2n}, \omega, g, \mu)$. This paper studies minimal submanifolds of toric K\"ahler manifolds, with a special emphasis on torus invariant minimal Lagrangians tori. The problems we will address here are related to the properties of such minimal tori appearing as the fibres of the moment map on a toric K\"ahler manifold. How many such minimal Lagrangian fibres exist? Can their stability, as critical points of the area functional, be characterised from the ambient geometry? Given a fixed toric manifold, to what extent can one prescribe the fibres of the moment map which are minimal with respect to some compatible toric K\"ahler metric?


\subsection{Summary of main results}

We will set up our problems in so-called action-angle coordinates. Our first result yields the existence of minimal Lagrangian torus on compact toric K\"ahler manifolds as well as a geometric criteria for uniqueness.

\begin{theorem}[Existence and criteria for uniqueness]\label{main}
Let $(M^{2n},\omega, g, \mu)$ be a compact toric K\"ahler manifold. 
\begin{enumerate}
\item[(a)] There is at least one Lagrangian orbit of the $\bbT^n$-action which is minimal. This corresponds to a null-homologous minimal Lagrangian $n$-torus.
\item[(b)] If the Ricci curvature of $g$ is everywhere positive, there is a unique orbit of the $\bbT^n$-action which is a minimal Lagrangian torus.
\end{enumerate}
\end{theorem}

\begin{remark}
	The existence part (a) in this Theorem already appears in \cite{Pacini}. As for the uniqueness criteria in (b) of the above statement: to our knowledge, the best result in this direction only considered the much more restrictive setting in which the metric is K\"ahler--Einstein (see \cite{Goldstein}).
\end{remark}

\begin{remark}
	An analogue of (b) in Theorem \ref{main} does not hold under the weaker assumption of everywhere positive scalar curvature. Indeed, in Remark \ref{rem:Any_Number_of_minimal_Lagrangian_tori} we shall give examples of positive scalar curvature toric K\"ahler metrics admitting any number of minimal torus orbits.
\end{remark}

As it turns out, the instability of minimal Lagrangian fibres may be inferred solely from the local ambient geometry. More concretely, Ricci and scalar curvatures of the ambient metric determine whether a minimal Lagrangian $\bbT^n$-orbit can be a local maximum/minimum of $V$ (the area functional restricted to the $\bbT^n$-orbits). {Note that because of symmetry the Ricci curvature is constant along a $\bbT^n$-orbit.}

\begin{theorem}[Stability criteria]\label{main_2}
	Let $(M^{2n},\omega, g,\mu)$ be a toric K\"ahler manifold. Then, the following hold true:
	\begin{enumerate}
		\item[(c)] If the Ricci curvature at a minimal Lagrangian $\bbT^n$-orbit $L$ is positive (respectively negative) then $L$ is a local maximum (respectively minimum) of the area functional restricted to the $\bbT^n$-orbits. In either case, $L$ is isolated as a minimal Lagrangian $\bbT^n$-orbit.
		\item[(d)] If the scalar curvature at a minimal Lagrangian $\bbT^n$-orbit is positive, then this orbit is not a local minimum of the area functional restricted to the $\bbT^n$-orbits and is therefore unstable. 
	\end{enumerate}
\end{theorem}

Recall that at a minimal submanifold $L\subset M,$ the second variation of the volume functional gives a linear operator on infinitesimal variations, i.e. normal vector fields (which in the Lagrangian case may be identified with $1$-forms on the submanifold). The number of negative eigenvalues of the second variation is called the index of the corresponding minimal submanifold and denoted $\mathrm{ind}(L)$.
We obtain an expression for the Ricci curvature which, at a critical point of $V$, identifies some of its components with a certain Hessian of the function $\log V$. From this we can read lower bounds on the index of minimal Lagrangian fibres simply from the Ricci curvature. Below we shall only state a crude version of our actual result, stated as Corollary \ref{cor:Index_Bounds}, which uses less restrictive assumptions on the Ricci curvature than the version stated here.

\begin{corollary}[Index estimate]\label{cor:Index_Bounds_Introduction}
	Let $(M^{2n},\omega,g,\mu)$ be a toric K\"ahler manifold and $L$ a minimal Lagrangian orbit around which the Ricci curvature is positive, then 
	$$\mathrm{ind}(L) \geq n.$$ 
\end{corollary}

As in \cite{Pacini}, we are also able to identify the mean curvature flow for these Lagrangians with a gradient flow for a real valued function on the moment polytope $P$.
 
\begin{theorem}\label{thm:Mean_Curvature}
	Let $(M^{2n},\omega,g,\mu)$ be a toric K\"ahler manifold and $x \in \mathrm{int} P$. Then, the mean curvature flow starting at $\mu^{-1}(x)$ coincides, via $\mu$, with the negative $g_P$-gradient flow of $\log V$ starting at $x$. 
\end{theorem}
Here the metric $g_P$ is the metric on the interior of the moment polytope induced from the toric K\"ahler metric. 

\begin{remark}
	This result was already known and contained in \cite{Pacini}. Notice that one of its consequences is that on a toric K\"ahler manifold, the mean curvature flow starting at a Lagrangian $\bbT^n$-orbit preserves the symmetry and thus the Lagrangian condition along the flow. In general, this need not be so although it is known that the mean curvature flow preserves the Lagrangian condition for K\"ahler-Einstein manifolds (see \cite{o}). 
\end{remark}

In the non-compact case, we use a result of Maccheroni ( see \cite{maccheroni}), ruling out the existence of certain holomorphic disks on manifolds with negative Ricci curvature to show the following.
\begin{theorem}\label{no_lag_nc_nr}
Let $(M^{2n},\omega,g,\mu)$ be a non-compact toric K\"ahler manifold whose moment map is proper and whose moment polytope has at least one vertex and a finite number of edges. Suppose the complex structure associated with $(\omega,g)$ is complete and the Ricci curvature of $g$ is negative. Then, there is no minimal Lagrangian orbit.
\end{theorem}
We will explain the notion of complete complex structure in section 6.  This is not equivalent to the completeness of $g.$\\

Finally, we raise a few questions (Questions \ref{prescribe_minimal_lagrangians}--\ref{prescribe_finite}) that follow naturally from our results. These concern the problem of prescribing minimal Lagrangians arising as fibres of the moment map on a toric manifold. More, precisely we would like to characterise those subsets of the moment polytope that occur as sets of minimal fibres for some toric K\"ahler metric on a given toric manifold. After looking at a few examples we prove the following partial answer to a version of this problem.
	
\begin{theorem}
	Let $X$ be a toric K\"ahler manifold with moment polytope $P$ and $p_1,\cdots, p_k$ distinct points in the interior of $P$. Then, there is a toric K\"ahler metric on $X$ such that the fibres over $p_1,\cdots, p_k$ are minimal Lagrangians.
\end{theorem}

This still does not completely answer the problem of prescribing a finite set of minimal Lagrangians because we cannot guarantee that $p_1,\cdots, p_k$ are the {\it only} minimal Lagrangians for the constructed metric. 

We finish our paper with a discussion of the Hsiang-Lawson principle applied to our setting as a potential source of higher dimensional minimal submanifolds on toric K\"ahler manifolds. In that direction we limit ourselves to constructing several simple examples of higher dimensional minimal submanifolds.


\subsection{Organisation} 
This paper is organised in the following way. Section 2 sets up the framework of action-angle coordinates to look for Lagrangian torus fibres which are critical for the area functional using Hsiang-Lawson's theorem. In section 3 we find curvature formulas in action-angle-coordinates which we later use, in section 4, to give our geometric criteria for uniqueness and stability of minimal Lagrangian fibres. In section 5 we study examples of such minimal Lagrangians, in view of which, we take some baby steps towards solving the problem of ``prescribing minimal Lagrangian fibres". We make some more progress towards a general solution of this problem in section 7. In section 6 we prove and discuss Theorem \ref{no_lag_nc_nr}. We devote section 8 to studying the mean curvature Lagrangian flow in this toric context, show it is a gradient flow on the polytope and look at some examples. Finally, section 9 illustrates how one could use the Hsiang-Lawson principle in concrete cases to find torus invariant higher dimensional minimal submanifolds.

\subsection{Acknowledgements}
We would like to thank Miguel Abreu, Jason Lotay, Andr\'e Neves, Tommaso Pacini, Daniele Sepe and Renato Vianna for interesting conversations regarding this work. We are very thankful to Tommaso Pacini for having called our attention to \cite{Pacini}.

\section{Existence of minimal Lagrangian tori}

Let ${(M^{2n},\omega, g)}$ be a toric K\"ahler manifold with moment map $\mu: M \rightarrow \bbR^n$. When $\mu$ is proper, its image is called the moment polytope of the toric manifold $(M^{2n},\omega)$ and is known to be convex. It can be written in the form 
\[
P=\{x\in \bbR^n: x\ip \nu_i-\lambda_i \geq 0, \,\, i=1\cdots d\},
\]
where the $\nu_i$'s are primitive, inward pointing boundary normals to the facets of $P$ and the $\lambda_i$'s are real constants. There is an open dense subset of $M$ which is diffeomorphic to $P\times \bbT^n$ yielding action-angle coordinates $(x, \theta) \in \mathrm{int}P\times \bbR^n.$ What is more, it follows from \cite{a} that there exists a convex function $u$ on $P$ such that
$$g = u_{ij} \ dx^i  dx^j + u^{ij} \ d \theta^i d \theta^j,$$
with 
$$u_{ij}=\frac{\del^2 u}{\del x^i \del x^j},$$ 
and $u^{ij}$ the entries of the inverse of $(u_{ij})_{i,j=1}^n$. The function $u$ is called the symplectic potential of the metric $g$ and it satisfies a boundary condition depending on the polytope. To be more precise, setting
$$u_G=\tfrac{1}{2}\sum_{i=1}^d \left( l_i\log l_i-l_i\right),$$ 
where $l_i(x)=x\ip \nu_i-\lambda_i$ for $i=1,\cdots d$, then $u-u_G$ is smooth in a neighbourhood of $\bar{P}.$ The function $u_G$ is the symplectic potential of the so-called Guillemin metric on $M.$

The area (or volume) of the $\mathbb{T}^n$-orbits is proportional to the orbital volume function 
$$V=\det(\Hess(u))^{-1/2},$$ 
which may be regarded as a function on the moment polytope $P$. From Palais{'} principle of symmetric criticality \cite{Palais} its interior critical points correspond to those $x \in \inte P$ so that $\pi^{-1}(x)\cong \bbT^n$ is a minimal submanifold. In the specific setting of minimal submanifolds, this principle was reformulated and established by Hsiang-Lawson (\cite{hl}) and will play a crucial role in our results.
\begin{theorem}[Hsiang-Lawson]
Let $G$ be a compact Lie group, $(M,g)$ a Riemannian manifold with an isometric $G$-action, and $N \subset  M$ a $G$-invariant submanifold. Then $N\subset M$ is minimal if the volume of $N$ is stationary with respect to all $G$-invariant variations of $N$.
\end{theorem}
In our setting, given $x\in \inte P$ the submanifold $\mu^{-1}(x)$ is obviously $\bbT^n$-invariant and all its $\bbT^n$-invariant variations lead to other interior $\mu$-fibres. It follows from the above theorem that

\begin{proposition}\label{prop:Correspondence_Minimal_Critical_V}
Let $(M^{2n},\omega,g,\mu)$ be a toric K\"ahler manifold. A Lagrangian torus fibre $\mu^{-1}(x)$ is a minimal submanifold if and only if $x \in \inte P$ is a stationary point of $V$.
\end{proposition}
It is now trivial to prove the existence part of Theorem \ref{main}
\begin{proof}[Proof of (a) in Theorem \ref{main}]
The orbital volume function $V: P \to \mathbb{R}$ is nonnegative and vanishes precisely at the boundary of $P$. In particular, if $M$ is compact so is $P$ and then $V$ must have an interior maximum. 
\end{proof}

\section{Ricci and scalar curvature of toric K\"ahler metrics}\label{ricci}

\subsection{Ricci curvature}


In this section we compute the Ricci and scalar curvatures of the toric K\"ahler metric $g$ in action-angle coordinates. The formulas will prove useful in establishing the uniqueness part of Theorem \ref{main} as well as in finding criteria for stability of the minimal Lagrangian orbits.

\begin{proposition}[Ricci Curvature in action-angle coordinates]\label{prop:Ricci}
Let $(M^{2n},\omega,g,\mu)$ be a toric K\"ahler manifold with action-angle coordinates $(x,\theta)$ and symplectic potential $u$. We have
	\begin{align}
	\mathrm{Ric} & =  - \sum_{i,j,k,l=1}^n \left(u_{il} \ \frac{\del}{\del x^j} \left( u^{lk} \frac{\del \log V}{\del x^k}  \right) \ dx^i  \otimes dx^j  - u^{ik} \frac{\del}{\del x^k} \left( u^{jl} \frac{\del \log V}{\del x^l}  \right) \ d\theta^i \otimes d \theta^j\right) ,
	\end{align}
	where $V= (\det \Hess(u))^{-1/2}$.
\end{proposition}

\begin{proof}
Recall from \cite{Mehdi} that in action-angle coordinates, the Ricci form $\rho$ of a toric K\"ahler metric is written as
\begin{equation}\label{eq:Ricci_form_rho}
\rho=-\frac{1}{2} \sum_{i,l,k=1}^n\frac{\del^2 u^{li}}{\del x^i \del x^k} dx^k \wedge d\theta^l.
\end{equation}
The Ricci curvature tensor may be obtained from this by the formula $\mathrm{Ric}(\cdot , \cdot)=-\rho(J \cdot , \cdot)$ where $J$ is the complex structure, which in the coordinates $(x,\theta)$ may be written as
$$J=\begin{pmatrix}
0 & - \Hess(u)^{-1} \\
\Hess(u) & 0 
\end{pmatrix}.$$
These formulae can be used together with the transpose of $J$ to conclude that
\begin{align}\nonumber
\mathrm{Ric} & = \frac{1}{2} \sum_{i,l,k,m=1}^n\frac{\del^2 u^{li}}{\del x^i \del x^k} \left( -u^{mk} d\theta^m \otimes d \theta^l - u_{ml}dx^m  \otimes dx^k \right) \\ \nonumber
& = - \frac{1}{2} \sum_{i,l,k,m=1}^n u_{ml} \frac{\del^2 u^{li}}{\del x^i \del x^k} \ dx^m  \otimes dx^k  - \frac{1}{2} u^{mk} \frac{\del^2 u^{li}}{\del x^i \del x^k} \ d\theta^m \otimes d \theta^l \\ \label{eq:Ricci_Intermediate}
& =  - \frac{1}{2} \sum_{i,l,k,m=1}^n u_{il} \frac{\del^2 u^{lk}}{\del x^k \del x^j} \ dx^i  \otimes dx^j  - \frac{1}{2} u^{ik} \frac{\del^2 u^{jl}}{\del x^l \del x^k} \ d\theta^i \otimes d \theta^j .
\end{align}
Furthermore, we now compute 
\begin{align*}
- \frac{1}{2} \sum_{k=1}^n\frac{\del^2 u^{lk}}{\del x^k \del x^j} & = - \frac{1}{2} \sum_{k=1}^n\frac{\del}{\del x^j} \left( \frac{\del u^{lk}}{\del x^k} \right) \\
& = - \frac{1}{2} \frac{\del}{\del x^j} \sum_{m,k,i=1}^n\left( \frac{\del u^{lk}}{\del x^i} u^{im} u_{mk} \right) \\
& =  \frac{1}{2} \frac{\del}{\del x^j} \sum_{m,k,i=1}^n \left( u^{lk} u^{im} \frac{\del u_{mk}}{\del x^i} \right) \\ 
& =  \frac{1}{2} \frac{\del}{\del x^j} \sum_{m,k,i=1}^n\left( u^{lk} u^{im} \frac{\del u_{im}}{\del x^k} \right) ,
\end{align*}
for all $l,j=1,\cdots n$ where we have used the identities $\sum_{k=1}^n(\del_i u^{lk}) u_{mk} +  u^{lk} (\del_k u_{im}) =0$ for all $i,k=1, \cdots n$ and the fact that the mixed partial derivatives commute. Now, we find
$$\sum_{i,m=1}^n u^{im} \frac{\del u_{im}}{\del x^k}=\tr \sum_{i,m=1}^n \left(u^{im} \frac{\del u_{mj}}{\del x^k} \right)_{ij} = \sum_{i,m=1}^n\frac{\del u_{im}}{\del x^k} \frac{\del}{\del u_{mi}} \left( \log \det H \right) = \frac{\del}{\del x^k} \log \det H ,$$
and inserting into the previous computation we find that for all $k=1,\cdots n$
\begin{align*}
- \frac{1}{2} \sum_{k=1}^n\frac{\del^2 u^{lk}}{\del x^k \del x^j} & =  \frac{1}{2} \frac{\del}{\del x^j} \sum_{k=1}^n\left( u^{lk} \frac{\del}{\del x^k} \log ( \det  \Hess(u) ) \right) \\
& =  - \frac{\del}{\del x^j} \sum_{k=1}^n\left( u^{lk} \frac{\del \log V}{\del x^k}  \right),
\end{align*}
for all $j,l=1\cdots l.$ Relabeling the indices we obtain
\begin{align}\label{eq:Rewritting_Ricci_Coefs}
- \frac{1}{2} \sum_{l=1}^n\frac{\del^2 u^{jl}}{\del x^l \del x^k} & =  - \frac{\del}{\del x^k} \sum_{l=1}^n\left( u^{jl} \frac{\del \log V}{\del x^l}  \right),
\end{align}
for all $k,j=1\cdots l.$ Thus, the Ricci form as in equation \ref{eq:Ricci_form_rho} can be rewritten as
\begin{equation}\label{eq:Ricci_form_rho_2}
\rho=- \sum_{i,l,k=1}^n \frac{\del}{\del x^k}\left( u^{jl} \frac{\del \log V}{\del x^l}  \right) dx^k \wedge d\theta^j.
\end{equation}
On the other hand, inserting \ref{eq:Rewritting_Ricci_Coefs} into equation \ref{eq:Ricci_Intermediate} we find that
\begin{align}\nonumber
\mathrm{Ric} & =  - \sum_{i,j,l,k=1}^nu_{il} \ \frac{\del}{\del x^j} \left( u^{lk} \frac{\del \log V}{\del x^k}  \right) \ dx^i  \otimes dx^j  - \sum_{i,j,l,k=1}^nu^{ik} \frac{\del}{\del x^k} \left( u^{jl} \frac{\del \log V}{\del x^l}  \right) \ d\theta^i \otimes d \theta^j .
\end{align}
\end{proof}

\begin{remark}[Hessian of $\log V$ from $\mathrm{Ric}$]\label{rem:Hessian}
	From the formula in Proposition \ref{prop:Ricci} it is clear that in the action directions, the entries of the Ricci tensor at a critical point of $V$ coincide with those of the Hessian of $\log V$.\\
	Furthermore, a toric K\"ahler metric $g = u_{ij} \ dx^i  dx^j + u^{ij} \ d \theta^i d \theta^j$ induces a metric on $\inte P$ by requiring that $\mu$ is a Riemannian submersion. This metric shall be denoted $g_P$ and written in action coordinates as $g_P = u_{ij} \ dx^i  dx^j.$ Using it to compute the Hessian of a function $f:P\to \mathbb{R}$ we find that
	\begin{align*}
	\Hess^{g_P}_{ij}f & = \frac{\partial^2 f}{\partial x^i \partial x^j} - (\nabla_{\partial_{x_i}} \partial_{x_j}) f \\
	& = \frac{\partial^2 f}{\partial x^i \partial x^j} - \frac{1}{2}u^{kl}\frac{\partial u_{li}}{\partial x^j} \frac{\partial f}{\partial x^k},
	\end{align*} 
	where we used the Koszul formula to compute $\nabla_{\partial_{x_i}} \partial_{x_j} = \tfrac{1}{2}u^{lk}\tfrac{\partial u_{il}}{\partial x^j} \tfrac{\partial }{\partial x^k}$.
	From this and the formula in Proposition \ref{prop:Ricci}, at a critical point of $V$, the action components of the Ricci tensor can be written simply in terms of the Hessian of $\log V$ with respect to the metric $g_P$, i.e. $\Hess^{g_P} \log V$. Indeed we have
	\begin{equation}\label{eq:Ricci_In_Terms_Of_Hessian}
	\mathrm{Ric}(\partial_{x_i},\partial_{x_j})= - \Hess^{g_p}_{ij} \log V .
	\end{equation}
\end{remark}

\subsection{Scalar curvature}

Here we state an auxiliary result derived from calculations in \cite{a} and \cite{a2} (see also \cite{w}).

\begin{proposition}[Abreu's formula]\label{prop:Laplacian_of_V}
	Let $(M^{2n},\omega, g, \mu)$ be a toric K\"ahler manifold and let $s:P \to \mathbb{R}$ be the scalar curvature of the K\"ahler metric seen as a function on $P$. Then $V=(\det \Hess(u))^{-1/2}$ satisfies the equation
	\begin{equation}\label{eq:Laplacian_of_V}
	\Delta_P V = sV,
	\end{equation}
	where $\Delta_P$ denotes the Laplacian of the metric $g_P=u_{ij}dx^i \otimes dx^j$ on $P$.
\end{proposition}
\begin{proof}
	We may compute the scalar curvature from our formula for the Ricci curvature or from Abreu's formula (see \cite{a})
	$$s=- \frac{1}{2} \sum_{i,j=1}^n \frac{\partial^2 u^{ij}}{\partial x^i \partial x^j}=  - \sum_{i,j=1}^n\frac{\del}{\del x^i} \left( u^{ij} \frac{\del \log(V)}{\del x^j} \right).$$
	Recalling that $V=\sqrt{\det \Hess(u)^{-1}}$ we find
	\begin{eqnarray}\nonumber
	s & = &  - \sum_{i,j=1}^n \frac{\del}{\del x^i} \left( u^{ij} \frac{\del \log(V)}{\del x^j} \right) \\ \nonumber 
	& = & -  \sum_{i,j=1}^n\frac{\del}{\del x^i} \left( u^{ij} \sqrt{\det \Hess(u)} \frac{\del V}{\del x^j} \right) \\ \nonumber  \\ \nonumber 
	& = &  - \sum_{i,j=1}^n \frac{1}{V} \frac{1}{\sqrt{\det \Hess(u)}}\frac{\del}{\del x^i} \left( u^{ij} \sqrt{\det \Hess(u)} \frac{\del V}{\del x^j} \right) \\ \nonumber
	& = &  \frac{\Delta_P V}{V} 
	\end{eqnarray}
	which gives the formula in the statement.
\end{proof}

An interesting consequence of equation \ref{eq:Laplacian_of_V} above, which is simply a rewrite of Abreu's formula for the scalar curvature, and the existence result in (a) of Theorem \ref{main} is the following well known result (see \cite{w}).

\begin{corollary}\label{cor:No_Compact_Negative_Scalar_Curvature}
	There is no compact toric K\"ahler manifold with everywhere nonpositive scalar curvature.
\end{corollary}
\begin{proof}
	We argue by contradiction and suppose we are given such a compact toric K\"ahler manifold with $s\leq0$. Then, it follows from equation \ref{eq:Laplacian_of_V} that $\Delta_P V \leq 0$ and so $V$ cannot have any interior maximum. On the other hand, as we have seen in the proof of the existence part of Theorem \ref{main}, the function $V \geq 0$ vanishes at the boundary of $P$ and is positive in its interior where it must therefore attain a local maximum. This gives a contradiction.
\end{proof}

\section{Ricci curvature and minimal Lagrangian tori}

In the proof of the existence part of Theorem \ref{main}, we show that one can find a minimal Lagrangian torus by finding a maximum of $V$. Of course, other critical points may or not exist and it is desirable to understand whether they do. If they do, we are interested in understanding whether they are locally minimising, maximising or saddles for $V$. In this section we use the calculations from the previous section \ref{ricci} to answer these questions.

\subsection{Uniqueness}

The uniqueness part of Theorem \ref{main} follows immediately from our formula for the Ricci curvature stated in Proposition \ref{prop:Ricci}.
\begin{proof}[proof of (b) in Theorem \ref{main}]
	As before we have a correspondence between minimal Lagrangian $\bbT^n$-orbits and interior critical points of $\log (V)$. If $\mathrm{Ric} >0$ then $M$ is compact from Myers' theorem and so there is at least one such critical point by the existence part of Theorem \ref{main}. At such a critical point we have $d \log V =0.$ Inserting this in the formula for the Ricci curvature from Proposition \ref{prop:Ricci} we find that
	\begin{align}\label{eq:Ricci_Critical_Point}
	\mathrm{Ric} & =  - \sum_{i,j=1}^n \frac{\del^2 \log V}{\del x^i \del x^j}  \ dx^i  \otimes dx^j  - \sum_{i,j,k,l=1}^nu^{ik}  u^{jl} \frac{\del^2 \log V}{\del x^k \del x^l}   \ d\theta^i \otimes d \theta^j .
	\end{align}
	In particular, as $\mathrm{Ric}>0$ by assumption, all these critical points are local maxima. It follows immediately from Morse theory for either $V$ or $\log V$ that there can be only one such. 
\end{proof}

\begin{remark}\label{rem:Any_Number_of_minimal_Lagrangian_tori}
	An analogue of the uniqueness part of Theorem \ref{main} does not hold under the weaker assumption of everywhere positive scalar curvature. Indeed, consider a toric metric on $S^2$ of the form $g_h =dt^2 + h^2(t) d\theta^2$ for $t \in [0,1].$ We may design the profile function $h(t)$ so as to have as many critical points as we want. These critical points correspond to geodesic orbits of a circle action. Thus, their product with the unique closed geodesic circle orbit of the standard round sphere $(S^2,g_r)$ yields any number of $\bbT^2$-invariant minimal Lagrangian tori in $(S^2 \times S^2, g_h \oplus g_r)$. Furthermore, by shrinking $g_r$ down by $0< \lambda \ll 1$ we can make the scalar curvature of $g_h \oplus \lambda g_r$ as large as we want.\\
	This construction gives examples of positive scalar curvature toric K\"ahler metrics admitting any number of minimal torus orbits.
\end{remark}

\subsection{From Ricci curvature to stability}

We may also use the formula for the Ricci curvature in Proposition \ref{prop:Ricci} or more directly the formula in equation \ref{eq:Ricci_Critical_Point} at a critical point of $V$ to infer on the stability of a minimal orbit. The instability criterium which is stated as part (c) of Theorem \ref{main_2} follows from this analysis.

\begin{corollary}[Item (c) of Theorem \ref{main_2}]\label{cor:(a)}
	Let $(M^{2n},\omega, g, \mu)$ be a toric K\"ahler manifold and $L=\mu^{-1}(x)$ for $x \in \inte P$ a minimal Lagrangian. If $\mathrm{Ric}>0$, respectively $\mathrm{Ric}<0$, at $x$, then $L$ locally maximizes, respectively minimizes $V$ amongst Lagrangian $\bbT^n$-orbits. Furthermore, if $\mathrm{Ric}$ is non-degenerate at $x$, then $L$ is isolated as a minimal Lagrangian orbit.	
\end{corollary}
\begin{proof}
	We start by recalling how to write $\Hess^{g_p} \log V$ in terms of the Ricci tensor as in Remark \ref{rem:Hessian}. At a critical point of $V$
	\begin{equation}\label{eq:Ricci_In_Terms_Of_Hessian}
	\mathrm{Ric}(\partial_{x_i},\partial_{x_j})= - \Hess^{g_p}_{ij} \log V .
	\end{equation}
	Denote by $\mathrm{Ric}_x$ the Ricci tensor at a critical point $x$ of $V$, or equivalently of $\log V$. From this formula it follows immediately that if $\mathrm{Ric}_x>0$ then $\log V$ has a local maximum at $x$, while if $\mathrm{Ric}_x<0$ this is a local minimum. In fact, it follows that if $\mathrm{Ric}$ is non-degenerate so is $\Hess^{g_P} \log V$. From the Morse lemma, $x$ is isolated as a critical point of $V$.
\end{proof}

Finally we address the index estimates which we presented in the Introduction as Corollary \ref{cor:Index_Bounds_Introduction}. Here we shall actually prove a finer statement which requires some preparation. The open dense set in $M$ which is diffeomorphic to $\inte P \times \mathbb{T}^n$ via action-angle coordinates may be regarded as a (trivial) $\mathbb{T}^n$-bundle over $\inte P$. The angle coordinates equip this bundle with a connection whose horizontal space $\mcH$ is the kernel of the $1$-forms $\lbrace d \theta^i \rbrace_{i=1}^n$. Then, the restriction of $\mathrm{Ric}$, the Ricci tensor of the K\"ahler metric $g$, to the horizontal space yields a quadratic form on each horizontal space. At a point $x \in \inte P$, the dimension of the largest subspace of $\mcH_x$ where Ricci is negative definite is called the Ricci-index and denoted by $\mathrm{ind_{Ric}}(x).$


\begin{corollary}[Finer version of Corollary \ref{cor:Index_Bounds_Introduction}]\label{cor:Index_Bounds}
	Let $x \in \inte P$ and $\mu^{-1}(x)$ be a minimal Lagrangian $\mathbb{T}^n$-orbit. Then, as a minimal submanifold 
	$$\mathrm{ind}(\mu^{-1}(x)) \geq (n-\mathrm{ind_{Ric}}(x) ).$$ 
	In particular, if $\mathrm{Ric}$ is positive at $x$, then $\mathrm{ind}(\mu^{-1}(x)) \geq n$.
\end{corollary}
\begin{proof}
	This again is an immediate consequence of the formula for Ricci tensor. Indeed from Remark \ref{rem:Hessian}, for $x \in \inte P$ a critical point of $V$ we find
	\begin{equation}\nonumber
	\mathrm{Ric}|_{H_x}= - (\Hess^{g_p} \log V)_x ,
	\end{equation}
	and given that $\mathrm{Ric}|_{H_x}$ and $\Hess^{g_p} \log V$ are quadratic forms in $n$-dimensions the result follows.
\end{proof}

\subsection{From scalar curvature to stability}

We can partially improve the results of the previous section to give a criterion for instability which only depends on the scalar curvature. This follows from the formula for the scalar curvature in Proposition \ref{prop:Laplacian_of_V} which reads
\begin{equation}\nonumber
\Delta_P V = sV,
\end{equation}
where $\Delta_P$ is the Laplacian of the metric $g_P=u_{ij}dx^i \otimes dx^j$ on $P$.
As a consequence of this equation, we deduce the following result which also proves the second part of Theorem \ref{main_2}.

\begin{corollary}[Item (d) of Theorem \ref{main_2}]
	Let $(M^{2n}, \omega , g, \mu)$ be a toric K\"ahler manifold, $s$ its scalar curvature, and $L=\mu^{-1}(x)$ for $x \in \mathrm{int P}$ a minimal Lagrangian torus. If $s > 0$ at $x$, then $L$ is unstable.
\end{corollary}
\begin{proof}
	At such a critical point $x$ of $V$ where $s(x)>0$ we find
	$$\Delta_P V > 0 , $$
	at $x$, which shows that $V(x)$ cannot be a local minimum.
\end{proof}

\section{Examples of minimal Lagrangian tori}

%
%
\begin{example}[$\mathbb{C}^2$]\label{ex:C2_Minimal}
	$l_1(x)=x_1$ and $l_2(x)=x_2$, thus $v_1=(1,0)$, $v_2=(0,1)$ and
	$$\Hess(u_G)=\begin{pmatrix}
	\frac{1}{2x_1} & 0 \\
	0 & \frac{1}{2x_2}
	\end{pmatrix}.$$
	We see that $V(x,y)= 2\sqrt{x_1 x_2}$ and this has no critical point in $(\mathbb{R}^+)^2$ which is the interior of $P$. Thus, there is no minimal Lagrangian orbit of $\mathbb{T}^2$ in $\mathbb{C}^2$.
\end{example}

\begin{example}[$\mathbb{CP}^2$]\label{ex:CP2_Minimal}
	The moment polytope of $\bcp^2$ is the standard simplex. We have $l_1(x)=x_1$, $l_2(x)=x_2$, $l_3(x)=1-x_1-x_2$ thus $v_1=(1,0)$, $v_2=(0,1)$ and $v_3=(-1,-1)$ and
	$$\Hess(u_G)=\frac{1}{2}\begin{pmatrix}
	\frac{1}{x_1} + \frac{1}{1-x_1-x_2} & \frac{1}{1-x_1-x_2} \\
	\frac{1}{1-x_1-x_2} & \frac{1}{x_2} +\frac{1}{1-x_1-x_2}
	\end{pmatrix}.$$
	So 
	\begin{align*}
	4\det(\Hess(u_G)) & = \left( \frac{1}{x_1} + \frac{1}{1-x_1-x_2} \right) \left(\frac{1}{x_2} + \frac{1}{1-x_1-x_2}\right) - \left(\frac{1}{1-x_1-x_2} \right)^2 \\
	& = \frac{1}{x_1x_2} + \frac{1}{1-x_1-x_2} \frac{1}{x_1} +  \frac{1}{x_2}  \frac{1}{1-x_1-x_2}\\
	& =  \frac{1}{x_1x_2(1-x_1-x_2)},
	\end{align*}
	and $V(x_1,x_2)=2\sqrt{x_1x_2(1-x_1-x_2)}$. This has unique interior critical point which is located at 
	$$x_1=\tfrac{1}{3}=x_2,$$ 
	and the corresponding minimal torus is known as the Clifford torus.
\end{example}

As stated in the introduction we are interested in understanding what characterises sets of minimal Lagrangian fibres. For instance, we would like to understand if such sets can accumulate or are restricted in some way. More precisely, we formulate the following question.

\begin{question}\label{prescribe_minimal_lagrangians}
	Let $X$ be a toric K\"ahler manifold with moment polytope $P.$ For which subsets $S$ of $\inte P$ is there a toric K\"ahler metric on $X$ whose minimal Lagrangian fibres are those fibres over $S$?
\end{question}

In the next examples we shall analyse this question and provide some partial answers for specific toric symplectic manifolds. These will provide intuition for a more thorough analysis in the next section.

\begin{example}[Cohomogeneity-1 metrics on $\mathbb{CP}^2$]
	Set $l_1(x)=x_1$, $l_2(x)=x_2$, $l_3(x)=1-x_1-x_2$ thus $v_1=(1,0)$, $v_2=(0,1)$ and $v_3=(-1,-1).$  A  metric whose symplectic potential is of the form $u(x)=u_G(x) + \frac{1}{2}f(x_1+x_2)$ has cohomogeneity-1. We have
	$$2\Hess(u)=\begin{pmatrix}
	\frac{1}{x_1} + \frac{1}{1-x_1-x_2} + f''& \frac{1}{1-x_1-x_2}+f'' \\
	\frac{1}{1-x_1-x_2}+f'' & \frac{1}{x_2} +\frac{1}{1-x_1-x_2}+f''
	\end{pmatrix} .$$
	So 
	\begin{align*}
	4\det(\Hess(u)) & = \left(\frac{1}{x_1} + \frac{1}{1-x_1-x_2}+ f''\right)\left(\frac{1}{x_1} + \frac{1}{1-x_1-x_2}+f''\right) - \left(\frac{1}{1-x_1-x_2} +f''\right)^2 \\
	& = \frac{1 + (x_1+x_2) (1-x_1-x_2) f''}{x_1x_2(1-x_1-x_2)}.
	\end{align*}
	The function $f$ needs to be such that $\Hess(u)$ is positive definite which is equivalent to the following system of inequalities
	$$
	\begin{cases}
	&1 + t (1-t) f''(t) >0,\\
	&\frac{1}{x_1} + \frac{1}{x_2}+\frac{2}{1-x_1-x_2} +2 f''>0,
	\end{cases}
	$$
	for $0\leq t \leq 1.$ It is not hard to check that the second condition above is equivalent to 
	$$
	f''>\frac{t-2}{t(t-1)}.
	$$
	Thus, we conclude that $\Hess(u)$ is positive definite if and only if
	$$
	f''> -\frac{1}{t(t-1)}.
	$$
	Now, recall that the minimal Lagrangian fibres we are looking for are in correspondence with critical points of $\det(\Hess(u))$ and thus of $\log \left( \det \Hess(u) \right)$ which is given by
	$$ 
	\log\left(1 + (x_1+x_2) (1-x_1-x_2) f''\right)-\log (x_1)-\log(x_2)-\log(1-x_1-x_2).
	$$ 
	Taking derivatives, we see that the minimal fibres lie above points satisfying the following system of equations
	$$
	\begin{cases}
	&\frac{(x_1+x_2) (1-x_1-x_2)f^{(3)} +\left( 1-2(x_1+x_2)\right)f''}{1 + (x_1+x_2) (1-x_1-x_2) f''}-\frac{1}{x_1}+\frac{1}{1-x_1-x_2}=0\\
        &\frac{(x_1+x_2) (1-x_1-x_2)f^{(3)} +\left( 1-2(x_1+x_2)\right)f''}{1 + (x_1+x_2) (1-x_1-x_2) f''}-\frac{1}{x_2}+\frac{1}{1-x_1-x_2}=0
        \end{cases}
        $$
        The above system implies that $x_1=x_2$ and setting $t=2x_1,$
        $$
        \frac{d}{dt}\log \left(1+t (1-t) f''\right)=\frac{2}{t}-\frac{1}{1-t}=\frac{2-3t}{t(1-t)}.
        $$
In the setting of $\bcp^2$ with cohomogeneity-1 toric K\"ahler metrics Question \ref{prescribe_minimal_lagrangians} becomes elementary and it is easy to prove the following results which we gather in the form of a proposition.
 \begin{proposition}
 \begin{itemize}
 \item[(a)] The minimal Lagrangian fibres of a cohomogeneity-1 toric K\"ahler metric on $\bcp^2$ sit above the diagonal $\bigtriangleup$ in the standard simplex.
 \item[(b)] A cohomogeneity-1 toric K\"ahler metric on $\bcp^2$ which is analytic admits a finite set of minimal Lagrangian fibres.
  \item[(c)] There is a (non-analytic) cohomogeneity-1 toric K\"ahler metric on $\bcp^2$ which admits a continuum of minimal Lagrangian fibres.
 \item[(d)] Given a finite subset $\{p_1,\cdots, p_k\}\subset \inte P\cap \bigtriangleup,$ there is a cohomogeneity-1 toric K\"ahler metric on $\bcp^2$ whose set of minimal Lagrangian fibres is {\it precisely} $\{p_1,\cdots, p_k\}.$
 \end{itemize}
 \end{proposition}
 \begin{proof} 
 \begin{itemize}
 \item[(a)] We have already proved this first item in the discussion preceding the statement of the proposition by noticing that 
 $$
	\begin{cases}
	&\frac{(x_1+x_2) (1-x_1-x_2)f^{(3)} +\left( 1-2(x_1+x_2)\right)f''}{1 + (x_1+x_2) (1-x_1-x_2) f''}-\frac{1}{x_1}+\frac{1}{1-x_1-x_2}=0\\
        &\frac{(x_1+x_2) (1-x_1-x_2)f^{(3)} +\left( 1-2(x_1+x_2)\right)f''}{1 + (x_1+x_2) (1-x_1-x_2) f''}-\frac{1}{x_2}+\frac{1}{1-x_1-x_2}=0
        \end{cases}
        $$
        implies $x_1=x_2.$
\item[(b)] Consider a cohomogeneity-1 toric K\"ahler metric on $\bcp^2$ with symplectic potential given by $u_G+\frac{1}{2}f(x_1+x_2).$ Set
$$
h(t)= \frac{d}{dt}\log \left(1+t (1-t) f''\right)-\frac{2-3t}{t(1-t)}.
$$
Given a set of minimal Lagrangians over $\{p_1=(x^1_1,x^1_2),\cdots, p_k=(x^k_1,x^k_2)\cdots \},$ setting $t_i=2x^i_1,$ the function $h$ must vanish exactly at $t_i.$ If the set $S$ is infinite it has an accumulation point. If the accumulation point is in the interior of the simplex, the corresponding set of t's $\{2x_1: (x_1,x_2)\in S\}$ has an accumulation point in $(0,1).$ So the zero set of $h$ in $]0,1[$ has an accumulation point which, because $h$ is analytic, forces $h$ to be zero on $]0,1[.$ But this is not possible as $f$ is smooth in a neighbourhood of the simplex. If the accumulation point is at $0$ or at $1$, then there is $t_i\rightarrow 0$ or  $t_i\rightarrow 1$ such that $h(t_i)=0$. But 
$$
\frac{d}{dt}\log \left(1+t (1-t) f''\right)_{|t_i}
$$
is bounded as $f$ is smooth in the closure of the simplex and 
$$
\frac{2-3t_i}{t_i(1-t_i)}\rightarrow \infty
$$
independently of whether $t_i\rightarrow 0$ or  $t_i\rightarrow 1$, so we get a contradiction as $h(t_i)=0.$ 
\item[(c)] Let $f$ be a smooth and non-negative function such that
$$
f''(t)=\begin{cases}
\frac{e^{\int_{\frac{1}{3}}^t\frac{2-3\tau}{\tau(1-\tau)}d\tau}-1}{t(1-t)}, t\in [1/3,2/3]\\
0, t\in [0,1/3-\epsilon]\cup[2/3+\epsilon, 1]\\
\end{cases}
$$
for some small $\epsilon>0.$ Furthermore, we can choose $f''$ so that it interpolates between these behaviours while remaining greater than $-\tfrac{1}{t(t-1)}$ so it defines a cohomogeneity-1 toric K\"ahler metric on $\bcp^2$. Furthermore, from construction we find that  
$$
\frac{d}{dt}\log \left(1+t (1-t) f''\right)-\frac{2-3t}{t(1-t)}=0
$$ 
in $[1/3,2/3]$ so that all the fibres above the diagonal segment between $(\frac{1}{6},\frac{1}{6})$ and  $(\frac{1}{3},\frac{1}{3})$ are minimal Lagrangians.
\item[(d)] Let $t_i=2x^i_1$ where $x^i_1$ is the first (or second) coordinate of $p_i$ for $i=1,\cdots k.$ Construct $\gamma,$ a smooth function on $[0,1]$ such that 
$$
\gamma-\frac{2-3t}{t(1-t)}
$$
vanishes exactly at $t_i$ for $i=1\cdots, k$ where $t_i=2x^i_1$ and such that $\int_0^1 \gamma(\tau)d\tau=0.$ Setting 
$$
f''=\frac{e^{\int_0^t \gamma(\tau)d\tau}-1}{t(1-t)}.
$$
We see that $f''$ is greater than $-\frac{1}{t(t-1)}$, $f''$ is smooth at $0$ and $1$ and
$$
\left(\prod_{i=1}^3 l_i \right)\det \left(\Hess \left (u_G(x) + \frac{1}{2}f(x_1+x_2) \right)\right)
$$
is smooth over $P$. It follows from this, by results in \cite{a2}, that $u_G(x) + \frac{1}{2}f(x_1+x_2)$ is the symplectic potential of a cohomogeneity-1 toric K\"ahler metric on $\bcp^2.$ On the other hand 
$$
\frac{d}{dt}\log \left(1+t (1-t) f''\right)-\frac{2-3t}{t(1-t)}=\gamma(t)-\frac{2-3t}{t(1-t)},
$$
so it vanishes exactly at $t_i$ for $i=1\cdots, k.$ Therefore the minimal Lagrangian fibres are exactly those above $\frac{1}{2}(t_i,t_i),$ for $i=1\cdots, k.$
\end{itemize}
\end{proof}

\begin{remark} 
We would have been able to carry out similar calculations for symplectic potential of the form $u_G+\frac{1}{2}f(x)$ which have minimal Lagrangian fibres  above the segment 
$$
\left\{(x,y):y=\frac{1-x}{2}\right\}.
$$
\end{remark}  
   \end{example}

\begin{example}[General toric metrics on $\mathbb{CP}^2$]
In fact for $\bcp^2$ we are able to say a bit more about Question \ref{prescribe_minimal_lagrangians} for general toric K\"ahler metric. Namely we prove the following:
\begin{proposition}
Let $p_1, \cdots, p_k$ be $k$ distinct points in the interior of the simplex in $\bbR^2.$ There is a toric K\"ahler metric whose fibres over $p_1, \cdots, p_k$ are minimal.
\end{proposition}
Note that we do not ensure there are no other minimal fibres.
\begin{proof}
We are going to look for toric K\"ahler metrics whose symplectic potential is of the form $u=u_G+\frac{1}{2}(f(x_1)+g(x_2)).$ In this setting
$$2\Hess(u)=\begin{pmatrix}
	\frac{1}{x_1} + \frac{1}{1-x_1-x_2} + f''& \frac{1}{1-x_1-x_2}\\
	\frac{1}{1-x_1-x_2}& \frac{1}{x_2} +\frac{1}{1-x_1-x_2}+g''
	\end{pmatrix}, $$
so that 
\begin{align*}
	4\det(\Hess(u)) & = \left(\frac{1}{x_1} + \frac{1}{1-x_1-x_2}+ f''\right)\left(\frac{1}{x_2} + \frac{1}{1-x_1-x_2}+g''\right) - \left(\frac{1}{1-x_1-x_2} \right)^2 \\
		& = \frac{1 +  f''x_1(1-x_1)+g''x_2(1-x_2)+f''g''x_1x_2(1-x_1-x_2)}{x_1x_2(1-x_1-x_2)}.
\end{align*}
	Thus, for $\Hess (u)$ to be positive definite in the interior of the simplex, the functions $f$ and $g$ need to satisfy
	\begin{equation}
        \begin{cases}\label{positive_definite}
	1 +  f''x_1(1-x_1)+g''x_2(1-x_2)+f''g''x_1x_2(1-x_1-x_2)>0,\\
	\frac{1}{x_1} + \frac{1}{x_2} + \frac{2}{1-x_1-x_2}+f''+g''>0,
	\end{cases}
	\end{equation}
whenever $x_1>0$, $x_2>0$ and $1-x_1-x_2>0$. As before, the minimal Lagrangian fibres sit above the critical points of 
$$
\log \left(1 +  f''x_1(1-x_1)+g''x_2(1-x_2)+f''g''x_1x_2(1-x_1-x_2)\right)-\log x_1-\log x_2-\log(1-x_1-x_2),
$$
and must therefore satisfy
\begin{IEEEeqnarray}{l}
\frac{(1-2x_1)f''+x_1(1-x_1)f^{(3)}+x_2(1-2x_1-x_2)f''g''+x_1x_2(1-x_1-x_2)f^{(3)}g''}{1 +  x_1(1-x_1)f''+x_2(1-x_2)g''+x_1x_2(1-x_1-x_2)f''g''}\nonumber\\
=\frac{1}{x_1}-\frac{1}{1-x_1-x_2}\nonumber,
\end{IEEEeqnarray}
and
\begin{IEEEeqnarray}{l}
\frac{(1-2x_2)g''+x_2(1-x_2)g^{(3)}+x_1(1-2x_2-x_1)f''g''+x_1x_2(1-x_1-x_2)g^{(3)}f''}{1 +  x_1(1-x_1)f''+x_2(1-x_2)g''+x_1x_2(1-x_1-x_2)f''g''}\nonumber\\
=\frac{1}{x_2}-\frac{1}{1-x_1-x_2}.\nonumber
\end{IEEEeqnarray}
Write the points $p_i$ as $p_i=(x_1^i,x_2^i)$, for $i=1,\cdots, k$, and choose $f''$ and $g''$ positive so that $f''(x^i_1)=g''(x^i_2)=1$. Then, at any of the points $p_i$, the above equations become
\begin{IEEEeqnarray}{l}\label{3rd_derivatives_1}
f^{(3)}\left( x_1(1-x_1)+x_1x_2(1-x_1-x_2) \right)=\\
\left(\frac{1}{x_1}-\frac{1}{1-x_1-x_2}\right)\left(1 + x_1(1-x_1)+x_2(1-x_2)+x_1x_2(1-x_1-x_2)\right)\nonumber\\
-(1-2x_1)-x_2(1-2x_1-x_2)\nonumber
\end{IEEEeqnarray}
and
\begin{IEEEeqnarray}{l}\label{3rd_derivatives_2}
{g^{(3)}\left(x_2(1-x_2)+x_1x_2(1-x_1-x_2)\right)}= \\
\left(\frac{1}{x_2}-\frac{1}{1-x_1-x_2}\right)\left(1 +  x_1(1-x_1)+x_2(1-x_2)+x_1x_2(1-x_1-x_2)\right)\nonumber\\
-(1-2x_2)-x_1(1-2x_2-x_1),\nonumber
\end{IEEEeqnarray}
which can be solved for $f^{(3)}$ and $g^{(3)}$ at each $p^i_1$. There exist smooth functions $f''$ and $g''$ on $[0,1]$ which are positive (they will then automatically satisfy condition \ref{positive_definite}), whose values at 
$x^i_1$ and $x^i_2$ respectively are $1$, and whose derivative at $x^i_1$ and $x^i_2$ respectively are given by the values prescribed by equations \ref{3rd_derivatives_1}--\ref{3rd_derivatives_2}. Because $f$ and $g$ can be chosen to be smooth up to the boundary of $P$
$$
\left(\prod_{i=1}^3 l_i \right)\det \left(\Hess \left (u_G(x) + \frac{1}{2}\left(f(x_1))+g(x_2)\right) \right)\right)
$$
is smooth in $P$ and therefore $u_G(x) + \frac{1}{2}\left(f(x_1))+g(x_2)\right)$ is the symplectic potential of a toric K\"ahler metric whose minimal Lagrangian fibres sit above the given points.
\end{proof}
\end{example}

\begin{example}[Guillemin metric on $\bcp^2\#\overline{\bcp^2}$]\label{one_blow_up}
The moment polytope of $\bcp^2\#\overline{\bcp^2}$ is
$$
P=\{(x_1,x_2)\in \bbR^2: \ x_1\geq 0, \ x_2\geq 0, \ 1-x_1-x_2\geq 0, \ x_2\leq a\}
$$
where $a$ in any constant in $(0,1).$ We have then 
$$2\Hess(u_G)=\begin{pmatrix}
	\frac{1}{x_1} + \frac{1}{1-x_1-x_2}& \frac{1}{1-x_1-x_2}\\
	\frac{1}{1-x_1-x_2}& \frac{1}{x_2} +\frac{1}{a-x_2}+\frac{1}{1-x_1-x_2}
	\end{pmatrix}, $$
	so that
$$
4\det(\Hess(u_G))=\frac{a-(x_2)^2}{x_1x_2(a-x_2)(1-x_1-x_2)}.
$$	
By taking the logarithm and differentiating as before we see that the minimal fibres are above points $(x_1,x_2)$ satisfying
$$
\begin{cases}
	&-\frac{1}{x_1}+\frac{1}{1-x_1-x_2}=0\\
        &-\frac{2x_2}{a-(x_2)^2}-\frac{1}{x_2}+\frac{1}{a-x_2}+\frac{1}{1-x_1-x_2}=0
        \end{cases}
        $$
        This implies that $x_1=\frac{1-x_2}{2}$ and
 $$
 2y^4-ay^3-5ay^2+a(3a+2)y-a^2=0,
 $$
 where we used $y$ to denote $x_2$. By directly studying the variations of the function $y\mapsto 2y^4-ay^3-5ay^2+a(3a+2)y-a^2,$ we can see that for all values of $a$ in $(0,1)$ it has a unique zero corresponding to the unique minimal Lagrangian fibre. It would be interesting to give a more concrete combinatorial interpretation of this special point in $P.$
\end{example}

\begin{example}[Guillemin metric in real 4-dimensions]
We generalize the above example to all Guillemin metrics on toric manifolds of complex dimension $2$. Even though we have no concrete combinatorial interpretation for the minimal Lagrangians for the Guillemin metric on a toric surface so far we give a characterisation of these special fibres. Let $P$ be a Delzant polytope in $\bbR^2$
$$
P=\{(x_1,x_2)\in \bbR^2: \ (x_1,x_2)\ip \nu_i-\lambda_i\geq 0, \, \ \ i=1,\cdots, d\},
$$
where $\nu_i$ are primitive interior normals to $P$ and set $l_i(x_1,x_2)=(x_1,x_2)\ip \nu_i-\lambda_i$. The symplectic potential for the Guillemin metric is 
$$
u_G=\frac{1}{2}\sum_{i=1}^d\left(l_i\log (l_i)-l_i\right),
$$
so that 
$$
2\Hess(u_G)=\sum_{i=1}^d \frac{\nu_i \nu_i^t}{l_i}.
$$
Minimal Lagrangian fibres for the Guillemin metric are critical points for $\log(\det(\Hess(u_G)))$ and using Jacobi's formula for the derivative of the determinant of a matrix we see that
\begin{align*}
\frac{\partial\log(\det(\Hess(u_G)))}{\partial x_l}&=\tr\left(\frac{\partial \Hess(u_G)}{\partial x_l}\Hess^{-1}(u_G)\right)\\
&=\frac{1}{\det(\Hess(u_G))}\tr\left(\frac{\partial \Hess(u_G)}{\partial x_l}\Hess^{*}(u_G)\right),
\end{align*}
where $\Hess^{*}(u_G)$ is the adjugate of $\Hess(u_G).$ Now 
$$
2\frac{\partial \Hess(u_G)}{\partial x_l}=-\sum_{i=1}^d \frac{\nu_i \nu_i^t}{(l_i)^2}(\nu_i)_l,
$$
where $(\nu_i)_l$ is the $l$th component of $\nu_i.$ In dimension $2,$ adjucation is linear and 
$$
2\Hess^{*}(u_G)=\sum_{i=1}^d \frac{\left(\nu_i \nu_i^t\right)^*}{l_i},
$$
where
$$
\left(\nu_i \nu_i^t\right)^*=\begin{pmatrix}
((\nu_i)_2)^2&-(\nu_i)_1(\nu_i)_2\\
-(\nu_i)_1(\nu_i)_2&((\nu_i)_1)^2
\end{pmatrix}=\mu_i \mu_i^t,
$$
where $\mu_i = ((\nu_i)_2,-(\nu_i)_1)$ is orthogonal to $\nu_i.$ Replacing in the Jacobi formula above we see that the minimal Lagrangian fibres for the Guillemin metric are those sitting above points which satisfy
\begin{align*}
0&=\tr\left(\frac{\partial \Hess(u_G)}{\partial x_l}\Hess^{*}(u_G)\right)\\
&=-\tr\left(\sum_{i,j=1}^d \frac{\nu_i \nu_i^t\mu_j \mu_j^t}{(l_i)^2l_j}(\nu_i)_l\right)\\
&=-\sum_{i,j=1}^d \nu_i\ip\mu_j\frac{\tr(\nu_i  \mu_j^t)}{(l_i)^2l_j}(\nu_i)_l\\
&=-\sum_{i,j=1}^d \frac{(\nu_i\ip\mu_j)^2}{(l_i)^2l_j}(\nu_i)_l,
\end{align*}
for $l=1,2.$ Now $\nu_i\ip\mu_j=\det(\nu_i,\nu_j)$ so we rewrite the above as 
$$
\sum_{i,j=1}^d \frac{\det^2(\nu_i,\nu_j)}{(l_i)^2l_j}\nu_i=0.
$$

\end{example}

\begin{example}[Standard $\mathbb{CP}^1 \times \mathbb{CP}^1$]\label{ex:CP1xCP1_Minimal}
	In this case $l_1(x)=x_1$, $l_2(x)=x_2$, $l_3(x)=1-x_1$ and $l_4(x)=1-x_2$, so
	$$\Hess(u_G) = \frac{1}{2} \begin{pmatrix}
	\frac{1}{x_1} + \frac{1}{1-x_1} & 0 \\
	0 & \frac{1}{x_2} + \frac{1}{1-x_2}
	\end{pmatrix}.$$
	From this we find
	$$\det(u_G)= \frac{1}{4} \frac{1}{x_1 x_2 (1-x_1) (1-x_2)} ,$$
	so $V(x_1,x_2)=2 \sqrt{x_1 x_2 (1-x_1) (1-x_2)}$. To locate its critical points we compute 
	\begin{align*}
	\frac{\del V}{\del x_1} & = \frac{2x_2 (1-x_2) ( 1-2x_1 )}{V} \\
	\frac{\del V}{\del x_1} & = \frac{2x_1 (1-x_1) ( 1-2x_2 )}{V},
	\end{align*}
	and so there is a unique interior critical point which is 
	$$x_1=\frac{1}{2} =x_2.$$ 
	By an abuse of nomenclature we shall also call the corresponding minimal torus the Clifford torus in $\mathbb{CP}^1 \times \mathbb{CP}^1$.
\end{example}

\section{Non-existence of minimal Lagrangian fibers}

The results of the previous section show that any toric K\"ahler manifold $(M,\omega,g,\mu)$ without minimal Lagrangian orbits must be non-compact. In this section we show that, in fact, for a large class of toric K\"ahler manifolds of negative Ricci curvature there cannot be a minimal Lagrangian orbit. One might hope that a Morse theoretic argument for the function $\log V$ could yield this result. However, such arguments are hard to apply without fixing an asymptotic behavior for the metric to control the level sets of $\log V$ along the ends. Hence, we shall instead follow an alternative strategy which seeks to rule out the existence of such minimal Lagrangian tori by using a recent non-existence theorem for holomorphic disks due to Maccheroni. We will construct holomorphic disks by taking advantage of the description of toric manifolds as K\"ahler quotients  of $\mathbb{C}^d.$ Compact toric manifolds can always be understood as such but this isn't so in the non-compact case. 

In \cite{kl}, Karshon and Lerman study non-compact symplectic toric manifolds. They prove the following theorem.
\begin{theorem} \label{thm:genDel}
Let $(M,\omega)$ be a symplectic toric manifold, with proper moment map $\mu$. Then $P=\mu(M)$ is a moment polygon.

Two symplectic toric manifolds are equivariant symplectomorphic (with respect to fixed torus actions) if and only if their associated moment polygons are isomorphic. 

Moreover, every moment polygon with a finite number of facets and at least one vertex arises from some symplectic toric manifold.
\end{theorem}
To prove the last statement in the above theorem, Karshon-Lerman use a construction via K\"ahler reduction of $\bbC^d$ where $d$ is the number of facets. This construction gives a reference compatible, invariant complex structure on $M$ and thus a metric which is just as the Guillemin metric from \cite{gu} in the compact case and which we therefore denote by $J_G.$

In the compact case, it is also known that any other compatible torus invariant complex structure is equivariantly biholomorphic to $J_G$ although the biholomorphism does not in general preserve the symplectic form.  A sketch of this can be found in \cite{a2}. For non-compact manifolds the situation is more delicate as we can see from the example of $\bbR^2$. The toric manifold $\bbR^2$ carries two inequivalent holomorphic structures which are torus invariant. Namely one biholomorphic to $\bbD,$ the other to $\bbC$ corresponding to the hyperbolic and flat metrics respectively. We are interested only in those complex structures which are biholomoprhic to $J_G$. These are defined in \cite{as}.

\begin{definition} \label{def:Jcomplete}
Let $(M,\omega)$ be a symplectic toric manifold. A toric compatible torus invariant complex structure $J$ is said to be {complete} if the $J$-holomorphic vector fields 
$JY$ where  $Y$ is any Hamiltonian vector fields generating the torus action, are complete. 
\end{definition}

It is not hard to see that the sketch of the proof of existence of an equivariant biholomorphism between compatible torus invariant complex structures extends to the case of complete 
compatible torus invariant complex structures on toric manifolds whose moment map is proper. This is discussed in \cite{as} for the case of surfaces but the arguments are not dimension dependent. Complete complete compatible toric complex structures on $(M,\omega)$ are thus biholomorphic to $J_G.$

\begin{theorem}
	Let $(M,\omega,\mu)$ be a toric K\"ahler manifold whose moment map is proper and whose moment map image has a finite number of facets and at least one vertex. Suppose $M$ is endowed with a complete complex structure whose metric is of strictly negative Ricci curvature. Then, there are no minimal Lagrangian orbits of the Hamiltonian torus action.
\end{theorem}
\begin{proof}
Recall Karshon-Lerman's generalisation of the Delzant construction, \cite{kl}, to good non-compact toric manifolds. Let $\nu_1, \ldots , \nu_d \in \mathbb{R}^n$ be primitive inwards pointing normals to the (non-compact) polytope $P$ and consider the exact sequence
$$0 \to \mathfrak{k} \to \mathbb{R}^d \to \mathbb{R}^n \to 0,$$
where the rightmost map sends a standard basis element $e_i$ to $\nu_i$. Given that the vectors $\nu_i$ are integral, the kernel $\mathfrak{k}$ can be exponentiated to form a torus $K \hookrightarrow \mathbb{T}^d$ which acts on $\mathbb{C}^d$ in the standard way. Denote by $\mu_K: \mathbb{C}^d \to \mathfrak{k}^*\cong \mathbb{R}^{d-n}$ the moment map for this action, then $(M,\omega)$ is obtained as the symplectic quotient of $\mathbb{C}^d$ by this action, for instance $M=\mu_K^{-1}(0)/ K$. The standard K\"ahler structure on $\mathbb{C}^d$ induces one on $M$, which corresponds to the Guillemin potential and the corresponding complex structure $J_G$. 

Recall that the any other invariant compatible complete complex structure $J$ on $M$ is equivariantly bi-holomorphic to $J_G$. Thus, $J$-holomorphic disks with boundary on a specific fibre exist if and only if they do for $J_G$. Thus, we shall now only search for holomorphic disks with respect to the canonical complex structure. In the compact case, these are classified by Cho and Oh in \cite{Cho} and always lift to holomorphic disks on $\mathbb{C}^d$. We shall thus construct holomorphic disks on $M$ by constructing appropriate holomorphic disks in $\bbC^d.$ In order to do that it is convenient to exhibit $M$ as a GIT K\"ahler quotient: $M$ is the orbit space of stable orbits of the complexified group $K_\bbC$
$$
M= \left(\cup_{F \leq P} \bbC^d_F\right) /K_{\bbC},
$$
where union is taken over every face $F$ of $P$. Each such face is determined as the vanishing of a set of affine functions
$$F=\{x\in P:  l_i(x)=0, i\in I_F\subset \{1,\cdots,d\}\, l_i(x)\ne 0, i\notin I_F\}.$$ 
We define $\bbC^d_F$ to be 
$$
\bbC^d_{F}=\{(z_1,\cdots z_d):  z_i\ne 0 \text{ if } i\notin I_F, \text{ and } z_i=0 \text{ if } i\in I_F\}.
$$
Now, we show that for any Lagrangian orbit $L=\mu^{-1}(x)$ there is at least one holomorphic disk with boundary on $L$. Assume that $x$ is in the interior of a face $F$. Then $\mu^{-1}(x)$ is contained in $\bbC^d_F/K_\bbC.$ Let $(z_1,\cdots,z_d)$ be in $\pi^{-1}(x)$ and assume, without loss of generality, that $1\notin I_F.$ Note that no face $F$ can have $I_F=\{1,\cdots, d\}.$ Now consider 
$$
\begin{aligned}
w:\bbD^2 &\mapsto \mathbb{C}^{d}_F\\
z&\rightarrow (zz_1,\cdots,z_d). 
\end{aligned}
$$
and the composition of this map with the quotient map
$$
q: \left(\cup_{F \leq P} \bbC^d_F\right)\mapsto M= \left(\cup_{F \leq P} \bbC^d_F\right) /K_{\bbC},
 $$
$q\circ w.$ Now the action of $\bbT^n$ on $M$ is defined via an identification of $\bbT^d/K$ with $\bbT^n.$ When $|z|=1$, then $(zz_1,\cdots,z_d)=(z,1\cdots,1)\cdot (z_1,\cdots,z_d)$ so that $q(z_1,\cdots,z_d)$ is in the same $\bbT^n\simeq\bbT^d/K$ orbit as $q(zz_1,\cdots,z_d)$ and
$$
\mu(q(z_1,\cdots,z_d))=\mu(q(zz_1,\cdots,z_d)),
$$
hence $q\circ w(\partial \bbD)\subset L=\mu^{-1}(x).$ We have thus shown that for any Lagrangian orbit $L=\mu^{-1}(x)$ there is at least one holomorphic disk with boundary on $L$. In \cite{maccheroni}, Maccheroni proves that a minimal Lagrangian in a K\"ahler manifold whose Ricci curvature is negative, cannot bound holomorphic disks. Therefore $L$ cannot be minimal with respect to a toric K\"ahler metric.
\end{proof}

\section{Prescribing minimal Lagrangians}\label{sec:Prescribing}

It is interesting to wonder what is special about minimal Lagrangian fibres. Although not much is known about this question in general, what we would like to do here is to understand what sets of fibres of the moment map on a toric manifold can be realised as minimal Lagrangians for some toric K\"ahler metric on the toric manifold. We have already formulated our question more precisely in the examples as Question \ref{prescribe_minimal_lagrangians}. We have also seen that the type of answer we get depends strongly on whether the metric we are considering is analytic as we proved the following:
\begin{proposition}
There are (non-analytic) toric K\"ahler metrics on $\mathbb{CP}^2$ which admit accumulating minimal Lagrangians.
\end{proposition}
We do not expect this to be so in general. For instance, we have seen in Corollary \ref{cor:(a)} that if $Ric$ is everywhere non-degenerate, then the minimal Lagrangian fibres must be isolated. With this in mind, we address a more refined version of Question \ref{prescribe_minimal_lagrangians}. We would like to know which finite sets arise as moment map images of minimal Lagrangian fibres for some toric K\"ahler metric. 

\begin{question}\label{prescribe_finite}
Let $X$ be a toric K\"ahler manifold with moment polytope $P.$ For which sets of $k$ distinct points $p_1,\cdots, p_k$ in the interior of $P$ is there a toric K\"ahler metric on $X$ whose minimal Lagrangian fibres are those over $S=\lbrace p_1,\cdots, p_k \rbrace$?
\end{question}

Although we are not able to settle this question we will formulate a result towards an answer mainly because its proof seems to point in the direction of a possible strategy to a complete answer.

\begin{theorem}
Let $P$ be a Delzant polytope and $p_1,\cdots, p_k$ distinct points in the interior of $P.$ There is a toric K\"ahler metric on the toric manifold corresponding to $P$ such that the fibres over $p_1,\cdots, p_k$ are minimal Lagrangians.
\end{theorem}

Unfortunately, in our construction nothing ensures there are no other minimal fibres.

\begin{proof}
The first step is to consider a positive function $F\in \mcC^\infty(P)$ with critical points at $p_1,\cdots, p_k.$ Say for instance $F(p)=\prod_{i=1}^kd^2(p,p_i)$ for $p\in P.$ Now let $\Omega$ be a proper open smooth domain of $P$ containing $p_1,\cdots, p_k$ in its interior and write down the Monge-Amp\`ere equation with boundary values
\begin{equation}\label{MA}
\begin{cases}
\det(\Hess(u_G+f))=F, \quad \text{in} \, \Omega\\
f_{|\partial \Omega}=0.
\end{cases}
\end{equation}
This equation was considered in section 8 of \cite{cns} where the authors prove that it has a solution $f \in \mcC^\infty(\bar{\Omega})$ such that $\Hess(u_G+f)$ is positive definite in $\bar{\Omega}.$ Now extend $f$ by zero to $P.$ This may not be smooth but by perturbing $f$ in a neighbourhood of $\partial \Omega$ we can get $v$ such that
\begin{itemize}
\item $v\in \mcC^\infty(P).$
\item $\Hess(u_G+v)$ is positive definite in $\bar{\Omega}.$ 
\item The function $v$ is zero near $\partial P$ so that $u_G+v$ has the right behaviour near $\partial P.$
\item $\det(\Hess(u_G+v))=F$ in some open smooth proper subdomain of $\Omega$ containing $p_1,\cdots, p_k$ so that $p_1,\cdots, p_k$ are still critical points for $\det\Hess(u_G+v)$ and therefore minimal fibres for the metric whose symplectic potential is $u_G+v.$
\end{itemize}
And this proves our result.
\end{proof}
We would like to end this section by proposing a general strategy to settle Question \ref{prescribe_finite}. As before write 
$$
P=\{x\in \bbR^n: l_i(x)\geq 0, \, i=1,\cdots, d\},
$$
where $l_i(x)=x\ip \nu_i-\lambda_i$ with $\nu_i$ being the exterior primitive normals to the polytope $P.$ Note that for any toric K\"ahler metric with symplectic potential $u_G+f,$
$$
\left( \prod_{i=1}^d l_i \right) \det \Hess(u_G+f) \in \mcC^{\infty}(P).
$$
\begin{itemize}
\item Find a function $F\in \mcC^\infty(P)$ such that $F$ has critical points exactly at $p_1,\cdots, p_k$ and such that 
$$
\varphi = \left( \prod_{i=1}^d l_i \right) F \in \mcC^{\infty}(P).
$$
\item
Now solve for 
\begin{equation*}
\begin{cases}
\det(\Hess(u_G+f))=F, \quad \text{in} \,  P\\
f_{|\partial P}=0.
\end{cases}
\end{equation*}
Because $u_G$ is not smooth up to the boundary in $P$ the results in \cite{cns} do not apply. But their continuity method to prove existence may extend. Consider the following family of equations
\begin{equation*}
\begin{cases}
\det(\Hess(u_G+f_t))=\frac{(1-t)\left(\prod_{i=1}^d l_i\right)\det (\Hess(u_G))+t\varphi}{\prod_{i=1}^d l_i}, \quad \text{in} \,  P\\
f_t |_{\partial P} =0.
\end{cases}
\end{equation*}
It admits $0$ as a solution when $t=0$ and by proving a priori estimates on solutions one may be able to show that the equation can be solved for $t=1.$ 
\end{itemize}

\section{Lagrangian mean curvature flow}

Let $(X,g)$ be a Riemannian manifold and $Y \subset X$ a smooth submanifold. The Lagrangian mean curvature flow of $Y$ in $X$ is defined to be the flow of $Y$ by its mean curvature. We may write $Y_t$ for $t \in [0,T)$ for such a flow which must then satisfy
\begin{equation}\label{eq:mean_curvature_flow}
\del_t Y_t = H_t,
\end{equation}
where $H_t$ is the mean curvature of $Y_t$, i.e. the trace of the second fundamental form of $Y_t$. In a local orthonormal framing $\lbrace e_i \rbrace_{i=1}^n$ of $Y$ this can be written as $H_t=\sum_{i=1}^n(\nabla_{e_i}e_i)^{\perp}$, where $\perp$ denotes the projection onto the normal directions to $Y_t$. In the case when $X$ is equipped with a compatible symplectic structure $\omega$ and $Y_0$ is a Lagrangian submanifold, then the evolved submanifolds $Y_t$ need not be Lagrangian. However, under certain conditions such as $(X,\omega,g)$ being K\"ahler-Einstein this turns out to be the case. In such a situation we refer to the flow in \ref{eq:mean_curvature_flow} as Lagrangian mean curvature flow. We shall now see that the mean curvature flow of a Lagrangian $\bbT^n$-orbit on a toric K\"ahler manifold gives rise to a Lagrangian mean curvature flow. In fact, one of a very simple nature which descends to a gradient flow on the polytope. This is the content of Theorem \ref{thm:Mean_Curvature} which, for convenience, we restate here in a slightly more precise way.


\begin{theorem}\label{thm:Mean_Curvature_Precise}
	Let $(M^{2n},\omega,g, \mu)$ be a toric K\"ahler manifold and equip the interior of the moment polytope with the metric $g_P=\sum_{i,j=1}^n u_{ij}dx^i \otimes dx^j$. The mean curvature of the Lagrangian orbits can be written as
	$$H=-\nabla \log V.$$ 
	Let $x \in \inte P$ and $L_t$ be the mean curvature flow of $L_0=\mu^{-1}(x)$ for $t \in [0,T)$. Then, $\mu(L_t)=x_t \in P$ is a solution of the negative $g_P$-gradient flow of $\log (V)$ starting at $x$.
\end{theorem}
\begin{proof}
	Let $\lbrace e_i \rbrace_{i=1}^n $ and $\lbrace n_i \rbrace_{i=1}^n $ be local orthonormal frames of the $\mu$-fibres and their normal bundle respectively. Further denote by $\lbrace e^i \rbrace_{i=1}^n $ and $\lbrace n^i \rbrace_{i=1}^n $ the dual coframes. Then, the fibres' mean curvature is
	$$H=\sum_{i=1}^n (\nabla_{e_i} e_i)^{\perp} = \sum_{j=1}^n  \left( \sum_{i=1}^n S_{n_j}(e_i,e_i) \right) n_j,$$
	where $S_{n_j}(e_i,e_i)=n^j(\nabla_{e_i}e_i)$. Here, the operator $\perp$ on vector fields along a submanifold denotes the orthogonal projection onto the normal bundle to that submanifold's tangent bundle. Now, define the fibrewise tangent bundle map $S_{n_j}^{\#}$ by $\langle S_{n_j}^{\#}(X), Y \rangle = S_{n_j}(X,Y)$ with respect to which we can rewrite
	$$H= \sum_{j=1}^n \left( \sum_{i=1}^n  \langle S_{n_j}^{\#}(e_i) , e_i \rangle \right) n_j =\sum_{j=1}^n \left( \sum_{i=1}^n e^i(S_{n_j}^{\#}(e_i)) \right) n_j,$$
	as the $\lbrace e_i \rbrace_{i=1}^n $ are orthonormal. On the other hand, by the invariance of the trace under basis change we may further rewrite this as
	\begin{equation}\label{eq:Mean_Curvature_Intermediate}
	H = \sum_{j=1}^n  \left( \sum_{i=1}^n  d\theta^i(S_{n_j}^{\#}(\partial_{\theta_i} )) \right) n_j.
	\end{equation}
	Now we compute, for each $j$, the matrix entries of $S_{n_j}^{\#}$ by writing $S_{n_j}^{\#}(\partial_{\theta_i})=\sum_{l=1}^n S^j_{l i} \partial_{\theta_l}$. Then, we find
	\begin{align*}
	\sum_{l=1}^n u^{kl} S^j_{li} & = \langle S_{n_j}^{\#}(\partial_{\theta_i}) , \partial_{\theta_k} \rangle  \\
	& = n^j (\nabla_{\partial_{\theta_i}} \partial_{\theta_k} ).
	\end{align*}
	Multiplying by $u_{rk}$ and summing over $k$ we find $S^j_{ri}=n^j(\sum_{k=1}^n u_{rk} \nabla_{\partial_{\theta_i}} \partial_{\theta_k} ),$ for $i,j,r=1,\cdots, n$. Inserting this into equation \ref{eq:Mean_Curvature_Intermediate} we find that
	\begin{align*}
	H & = \sum_{j=1}^n  \left( \sum_{i=1}^n d\theta^i ( \sum_{l=1}^nS^j_{li} \partial_{\theta_l} ) \right) n_j \\
	& = \sum_{i,j,k,l=1}^n  \left(   d\theta^i \left( n^j(u_{lk} \nabla_{\partial_{\theta_i}} \partial_{\theta_k} ) \partial_{\theta_l} \right) \right) n_j \\
	& = \sum_{i,j,k=1}^n  \left( n^j(u_{ik} \nabla_{\partial_{\theta_i}} \partial_{\theta_k} )  \right) n_j \\
	& =  \sum_{i,k=1}^n u_{ik}( \nabla_{\partial_{\theta_i}} \partial_{\theta_k})^{\perp},
	\end{align*}
	Now, to compute $\nabla_{\partial_{\theta_i}} \partial_{\theta_k}$, we find the Christoffel symbols of the Levi-Civita connection of $g$ in the framing $\lbrace \partial_{x_i},\partial_{\theta_i} \rbrace_{i=1}^n$. These can be easily derived from the Koszul formula to be given by
	$$
	\Gamma_{\theta_i\theta_k}^{x_l}=- \sum_{j=1}^n\frac{1}{2}u^{lj}\frac{\partial{u^{ik}}}{\partial x_j},
	$$
	whereas $\Gamma_{\theta_i\theta_k}^{\theta_l}=0.$ Hence
	$$
	\nabla_{\partial_{\theta_i}} \partial_{\theta_k}=-\frac{1}{2}\sum_{j,l=1}^nu^{lj}\frac{\partial{u^{ik}}}{\partial x_j}\frac{\partial}{\partial x_l}.
	$$
	Therefore
	\begin{align*}
	H & =\sum_{i,k=1}^nu_{ik}( \nabla_{\partial_{\theta_i}} \partial_{\theta_k})^{\perp}\\
	&=-\frac{1}{2}\sum_{i,j,k,l=1}^n u^{jl}\left(u_{ik}\frac{\partial{u^{ik}}}{\partial x_j}\right)\frac{\partial}{\partial x_l}.\\
	\end{align*}
	Now we can write 
	$$
	\frac{\partial{u^{ik}}}{\partial x_j}=-\sum_{p,q=1}^nu^{ip}u_{pqj}u^{qk},
	$$
	for $i,j,k=1,\cdots n$, where we have have used the notation
	$$
	u_{pqj}=\frac{\partial^3 u}{\partial x_p \partial x_q \partial x_j}.
	$$
	It follows that 
	$$
	\sum_{i,k=1}^n u_{ik}\frac{\partial{u^{ik}}}{\partial x_j}=-\sum_{q,k=1}^nu^{qk}u_{qkj}.
	$$
	On the other hand it follows from Jacobi's formula for the derivative of the determinant of a matrix that
	$$
	\frac{\partial \left(\det \Hess u\right)}{\partial x_j}=\det \left(\Hess u\right) \sum_{p,q=1}^nu^{pq}u_{pqj},
	$$
	and so
	$$
	\sum_{i,k=1}^nu_{ik}\frac{\partial{u^{ik}}}{\partial x_j}=- \frac{1}{\det \left(\Hess u\right)} \frac{\partial \left(\det \Hess u\right)}{\partial x_j}.
	$$
	Inserting in $H$ we find that
	\begin{align*}
	H &=\frac{1}{2}\sum_{j,l=1}^nu^{jl}\frac{\partial \left(\log\left(\det \Hess u\right)\right)}{\partial x_j}\frac{\partial}{\partial x_l}\\
	&=\frac{1}{2}\nabla \log\left(\det \left(\Hess u\right)\right) \\
	&=-\nabla\log V.
	\end{align*}
	Given that the map $\mu : \inte P \times \bbT^n \subset M \to \inte P$ is a Riemannian submersion with respect to the metrics $g$ on $M$ and $g_P$ on $P$ respectively, the assertion regarding the mean curvature flow of a $\mu$-fibre follows.	
\end{proof}

\begin{remark}
	It also follows from this formula for the mean curvature vector $H$ that minimal $\bbT^n$-orbits are in correspondence with interior critical points of $V$ as stated in Proposition \ref{prop:Correspondence_Minimal_Critical_V}.  
\end{remark}

In general the mean curvature flow of a Lagrangian submanifold need not preserve the Lagrangian condition. As a consequence of Theorem \ref{thm:Mean_Curvature_Precise} we conclude that, in the case of toric K\"ahler manifolds, the mean curvature flow preserves the $\bbT^n$-orbits and so the Lagrangian condition for these. 

\begin{corollary}
	Let $(M^{2n},\omega,g,\mu)$ be a toric K\"ahler manifold and $x \in \inte P$. Let $L_t$, for $t \in [0,T)$ be a solution to the mean curvature flow with $L_0=\mu^{-1}(x)$, then $L_t$ is Lagrangian for all $t \in [0,T)$.
\end{corollary}

\begin{remark}
The Maslov form of a Lagrangian fibre $L=\mu^{-1}(x)$ is defined by
\begin{align*}
\sigma:= \iota_H \omega & = \sum_{i,j,l=1}^n dx^i \left( -u^{jl}\frac{\partial \log V}{\partial x^l} \frac{\partial}{\partial x^j} \right) d\theta^i \\
& = - \sum_{j,l=1}^n u^{jl}\frac{\partial \log V}{\partial x^l} d\theta^j .
\end{align*}
Thus
\begin{align*}
d \sigma & = - \sum_{i,j,l=1}^n\frac{\del}{\del x^i} \left( u^{jl}\frac{\partial \log V}{\partial x^l}  \right) dx^i \wedge d\theta^j,
\end{align*}
which vanishes when restricted to any $\bbT^{n}$-orbit. Therefore, for all $x \in \mathrm{int P}$ the $1$-form $\sigma |_{\pi^{-1}(x)}$ is closed and so defines a cohomology class $[\sigma] \in \mathrm{H}^1(\pi^{-1}(x), \mathbb{R})$. 
On the other hand, the Ricci $2$-form, as computed in equation \ref{eq:Ricci_form_rho_2} is given by
\begin{equation*}
\rho=- \sum_{i,j,l=1}^n \frac{\del}{\del x^i}\left( u^{jl} \frac{\del \log V}{\del x^l}  \right) dx^i \wedge d\theta^j,
\end{equation*}
and so agrees with $\rho$, i.e. we have proved that
$$\rho=d\sigma.$$
We further point out that the class $\sigma$ in only defined in $\mu^{-1}(\partial P)$ and so the equation above does not imply that $\rho$ is globally exact.
\end{remark}

\begin{example}[Standard $\mathbb{CP}^1 \times \mathbb{CP}^1$]\label{eq:CP1xCP1_Flow}
	Recall from example \ref{ex:CP1xCP1_Minimal} that in this case $P=[0,1]^2$ and $V(x_1,x_2)=2 \sqrt{x_1 x_2 (1-x_1) (1-x_2)}$ and
	\begin{align*}
	\frac{\del V}{\del x_1} & = \frac{2x_2 (1-x_2) ( 1-2x_1 )}{V} \\
	\frac{\del V}{\del x_1} & = \frac{2x_1 (1-x_1) ( 1-2x_2 )}{V}.
	\end{align*}
	The gradient of a function $f$ on $P$ with respect to $g_P$ is
	$$\nabla^{g_P} f = 2x_1(1-x_1) (\del_{x_1}f) \del_{x_1} + 2x_2(1-x_2) (\del_{x_2}f)\del_{x_2} ,$$ 
	and applying this to $-2\log V$ we find that the negative gradient flow of $-2\log V$ solves the system of ODE's
	$$\dot{x}_i = -(1-2x_i),$$
	for $i=1,2$. This is separable and can easily be integrated to give
	$$x_i(t)=\frac{1}{2} + \frac{2x_0-1}{2} e^{4t}. $$
	In particular, each $x_i(t)$ meets either $0$ or $1$ in finite time and so the flow $(x_1(t),x_2(t))$ hits the boundary of the polytope in finite time. Thus, any such Lagrangian torus which is not that at $(1/2,1/2)$ evolves through mean curvature flow to either a circle, if $x_1(0) \neq x_2(0)$ and $x_1(0)\neq 1-x_{2}(0)$, or a point if either $x_1(0) = x_2(0)$ and $x_1(0)= 1-x_{2}(0)$. Notice that the first is the generic situation.
\end{example}

\section{Examples of other minimal submanifolds}

\subsection{The Hsiang-Lawson lifting principle}

In \cite{hl} Hsiang and Lawson propose a method of finding minimal submanifolds of a manifold with a group action. The idea, in favourable cases, is roughly to find submanifolds of the quotient endowed with a certain metric (the Hsiang-Lawson metric). For completeness, we restate their result.
	 
\begin{theorem}[Hsiang-Lawson] 
	Let $G$ be a Lie group acting by isometries on the manifold $(M,g)$. A $G$-invariant submanifold $N$ of $M$ is minimal with respect to $g$ iff $N^*/G$ is minimal in $M^*/G$ with respect to a certain metric which is conformal to the quotient metric.
\end{theorem}

In the theorem, $M^*$ is the subset of $M$ consisting of $G$-principal orbits and $N^*=N\cap M^*.$ In this section we are going to look at the Hsiang-Lawson method on toric manifolds in some very simple examples. This might be a direction worth pursuing for finding examples of minimal submanifolds in particular situations.

In the situation when $(M^{2n},\omega,g,\mu)$ is a toric K\"ahler manifold, the quotient space $M^*/G=\mathrm{int} P$ comes equipped with a metric $g_P$ so that $\mu:(M^*,g) \to (\mathrm{int} P, g_P)$ is a Riemannian submersion. Then, given $\gamma^k \subset P$ we find that $\mu^{-1}(\gamma)$ is a $(n+k)$-dimensional submanifold of $M$ whose volume can be computed from the Fubini theorem to be 
\begin{align*}
\mathrm{Vol}(\mu^{-1}(\gamma)) & = \int_{\mu^{-1}(\gamma)} \dvol_{g|_{\mu^{-1}(\gamma)}} \\
& = \int_{\gamma} \dvol_{g_P |_\gamma}(x) \ \int_{\mu^{-1}(x)} \dvol_{g|_{\mu^{-1}(x)}} \\
& = \int_{\gamma} V(x) \dvol_{g_P |_\gamma}(x) ,
\end{align*}
which we can identify with the volume of $\gamma \subset \inte P$ when we equip $\inte P$ with the metric $V^{2/k}g_P$. Having this in mind we shall now see a few examples in the case when $M$ is 4-dimensional and $\gamma$ is a curve.

\begin{example}[Minimal hypersurfaces in the standard $\mathbb{CP}^1 \times \mathbb{CP}^1$]
	From the Hsiang-Lawson principle, a geodesic in $P=[0,1]^2$ with respect to the metric $V^2g_P$ lifts, via $\mu$, to a minimal submanifold of $\mathbb{CP}^1 \times \mathbb{CP}^1$. This metric can be written in the standard action coordinates $x_1,x_2$ on $P$ as
	$$V^2g_P= 2x_2(1-x_2) dx_1^2 + 2x_1(1-x_1) dx_2^2 .$$
	From Koszul's formula we find
	\begin{align*}
	\langle \nabla_{\partial_{x_1}} \partial_{x_2} , \partial_{x_1} \rangle & =  \del_{x_2} (2x_2(1-x_2)) = 2(1-2x_2) \\
	\langle \nabla_{\partial_{x_1}} \partial_{x_2} , \partial_{x_2} \rangle & =  \del_{x_1} (2x_1(1-x_1)) = 2(1-2x_1) ,
	\end{align*}
	from which it follows that
	$$\nabla_{\partial_{x_1}} \partial_{x_2} =\frac{1-2x_2}{x_2(1-x_2)} \partial_{x_1} + \frac{1-2x_1}{x_1(1-x_1)} \partial_{x_2} ,$$
	together with a similar formula for $\nabla_{\partial_{x_2}} \partial_{x_1}$ obtained from switching $x_1$ and $x_2$. In the same way we find $\langle \nabla_{\partial_{x_1}} \partial_{x_1} , \partial_{x_1} \rangle=0$ and
	$$\langle \nabla_{\partial_{x_1}} \partial_{x_1} , \partial_{x_2} \rangle = -\partial_{x_2} (2x_2(1-x_2))=-2(1-2x_2), $$
	so that
	$$\nabla_{\partial_{x_1}} \partial_{x_1} = - \frac{1-2x_2}{x_1(1-x_1)} \partial_{x_2},$$
	and a similar formula for $\nabla_{\partial_{x_2}} \partial_{x_2}$. We shall write a parametrised curve on $P$ as $\gamma(t)=(\gamma_1(t),\gamma_2(t))$ so that $\dot{\gamma}=(\dot{\gamma}_1,\dot{\gamma}_2)$ and 
	\begin{align*}
	\nabla_{\dot{\gamma}} \dot{\gamma} & = \sum_{i=1}^2 \left( \ddot{\gamma}_i \partial_{x_i} + \dot{\gamma}_i^2 \nabla_{\partial_{x_i}} \partial_{x_i} \right)  + 2\dot{\gamma}_1\dot{\gamma}_2 \nabla_{\partial_{x_1}}\partial_{x_2} \\
	& = \left( \ddot{\gamma}_1 - \dot{\gamma}_2^2 \frac{1-2\gamma_1}{\gamma_2(1-\gamma_2)} \right) \partial_{x_1} + \left( \ddot{\gamma}_2 - \dot{\gamma}_1^2 \frac{1-2\gamma_2}{\gamma_1(1-\gamma_1)} \right) \partial_{x_2} \\
	& \ \ \ \ +  2\dot{\gamma}_1\dot{\gamma}_2 \left( \frac{1-2\gamma_2}{\gamma_2(1-\gamma_2)} \partial_{x_1} + \frac{1-2\gamma_1}{\gamma_1(1-\gamma_1)} \partial_{x_2} \right)\\
	& = \left( \ddot{\gamma}_1 - \dot{\gamma}_2^2 \frac{1-2\gamma_1}{\gamma_2(1-\gamma_2)} + 2\dot{\gamma}_1\dot{\gamma}_2  \frac{1-2\gamma_2}{\gamma_2(1-\gamma_2)} \right) \partial_{x_1} \\
	& \ \ \ \ + \left( \ddot{\gamma}_2 - \dot{\gamma}_1^2 \frac{1-2\gamma_2}{\gamma_1(1-\gamma_1)}  + 2\dot{\gamma}_1\dot{\gamma}_2 \frac{1-2\gamma_1}{\gamma_1(1-\gamma_1)}  \right) \partial_{x_2} .
	\end{align*}
	The curve $\gamma$ is minimal with respect to $V^2g_P$ if and only if either $\gamma \in \subset V^{-1}(0)$, or $\nabla_{\dot{\gamma}} \dot{\gamma}=0$. The first case, i.e. $V(\gamma)=0$, happens if and only if $\gamma$ is contained in the boundary of the polytope. Then, the pre-image by $\mu$ of a boundary face is of the form $\lbrace \ast \rbrace \times \mathbb{CP}^1$ and is a divisor in $\mathbb{CP}^1 \times \mathbb{CP}^1$, thus an obvious minimal submanifold. If $V(\gamma) \neq 0,$ $\nabla_{\dot{\gamma}} \dot{\gamma} =0$, which we may write as
	\begin{equation}\label{geodesics_square}
	\begin{cases}
	\ddot{\gamma}_1 & =  \dot{\gamma}_2^2 \frac{1-2\gamma_1}{\gamma_2(1-\gamma_2)} + 2\dot{\gamma}_1\dot{\gamma}_2  \frac{1-2\gamma_2}{\gamma_2(1-\gamma_2)}  , \\
	\ddot{\gamma}_2 & =  \dot{\gamma}_1^2 \frac{1-2\gamma_2}{\gamma_1(1-\gamma_1)}  + 2\dot{\gamma}_1\dot{\gamma}_2 \frac{1-2\gamma_1}{\gamma_1(1-\gamma_1)}  .
	\end{cases}
	\end{equation}
	We have the following obvious solutions: $\gamma_1=1/2$ is constant (or $\gamma_2=1/2$) in which case the equations reduce to $\ddot{\gamma}_2=0$ (respectively  $\ddot{\gamma}_1=0$). Thus $\gamma$ must be either $\lbrace \tfrac{1}{2} \rbrace \times [0,1]$ or $[0,1] \times \lbrace \tfrac{1}{2} \rbrace$ and $\mu^{-1}(\gamma)$ is a minimal $S^1 \times \mathbb{CP}^1$ or $\mathbb{CP}^1 \times S^1$ respectively. These minimal submanifolds are well known but they illustrate a method for obtaining more interesting solutions as we shall see below. Indeed, both these geodesics could have been obtained by the symmetries of $P$ which reflect it along these geodesics. Such reflections preserve $V^2g_P$ and so we could have immediately concluded that their fixed locus consists of geodesics.\footnote{The reflections along the diagonals also preserve $V^2g_P$ but their pre-image is not smooth. Indeed, it has two points which have neighbourhoods modelled on cones over a $2$-torus.} Furthermore, it may be possible to find other solutions of Equation (\ref{geodesics_square}) which give rise to interesting minimal submanifolds of $\mathbb{CP}^1 \times \mathbb{CP}^1.$
\end{example}

In the spirit of the example above where the geodesics for the Hsiang-Lawson metric on the square are the fixed-point set of some extra symmetry on the polytope we may try to apply Hsiang-Lawson lifting principle to ``obvious" geodesics on other symmetric polytopes. Indeed, consider the following examples:

\begin{itemize}
\item For $\bcp^2$, the moment polytope is a triangle with a symmetry $(x,y)\mapsto (y,x)$. For any cohomogeneity-$1$ metric on $\bcp^2$ the segment corresponding to the intersection of $$\{(x,y):x=y\}$$ with the triangle is a geodesic because it is a fixed-point set under an isometry. Note that the function on the triangle given by the volume of a fibre is invariant by $(x,y)\mapsto (y,x)$ if the metric has cohomogeneity-$1$. The lift of this segment has a singularity modelled on a cone over a 2-torus, so the method does not yield a submanifold in this case. This can however be smoothed by blowing up as we shall now see.

\item Consider $\bcp^1\#\overline{\bcp^1}$ with a cohomogeneity-$1$ metric. The moment polytope can be chosen to be
$$
P=\{(x_1,x_2)\in \bbR^2: x_1\geq 0, x_2\geq 0, a\leq x_1+x_2\leq 1\},
$$
where $a\in ]0,1[$. This admits a symmetry $(x,y)\mapsto (y,x)$. Therefore the segment
$$
P\cap \{(x,y):x=y\}
$$ 
is a geodesic for the Hsiang-Lawson metric and its lift is a minimal $S^1\times S^2$ in $\bcp^1 \# \overline{\bcp^1}$. 

\item A similar construction can be carried out for $\bcp^1\#3\overline{\bcp^1}$ with a cohomogeneity-$1$ metric. The moment polytope is
$$
P=\{(x_1,x_2)\in \bbR^2: x_1\geq 0, x_2\geq 0, a\leq x_1+x_2\leq 1, x_2\leq b, \, x_1\leq b\}.
$$
The lift of 
$$
P\cap \{(x,y):x=y\}
$$ 
is a minimal $S^1\times S^2$ in $\bcp^1 \# 3\overline{\bcp^1}$.
\end{itemize}



\end{document}